\documentclass{amsart}
\makeatletter
\def\@@and{}
\makeatother
\usepackage{mathrsfs}
\usepackage{amsfonts}

\usepackage{amsfonts, amsthm, amssymb}
\usepackage{amssymb,amsfonts,amsmath,color,
latexsym, epsfig,cite, psfrag,eepic,colordvi,enumitem}
\usepackage{amscd,graphics,url}
\usepackage{lineno}
\usepackage[hidelinks]{hyperref}
\usepackage[ruled,vlined]{algorithm2e}
\usepackage{xcolor}
\usepackage{hyperref}
\hypersetup{hidelinks}
\hypersetup{
    colorlinks=true,
    urlcolor=blue!60!black,
    linkcolor=blue!60!black,
    citecolor=blue!60!black
}

\SetKwInOut{Input}{Input}
\SetKwInOut{Output}{Output}
\SetKw{KwRet}{return}

\allowdisplaybreaks

\newtheorem{theorem}{Theorem}[section]
\newtheorem{lemma}[theorem]{Lemma}

\newtheorem{proposition}[theorem]{Proposition}

\newtheorem{definition}[theorem]{Definition}

\title{\textbf{Perfect matching in 4-partite 4-uniform hypergraphs}}
\author{Hongliang Lu}
\address{School of Mathematics and Statistics,
Xi'an Jiaotong University,
Xi'an 710049, China}
\email{luhongliang215@sina.com}

\author{Yan Wang}
\address{School of Mathematical Sciences, CMA-Shanghai,
Shanghai Jiao Tong University,
Shanghai 200240, China}
\email{yan.w@sjtu.edu.cn}

\author{Feihong Yuan}
\address{School of Mathematics and Statistics,
Xi'an Jiaotong University,
Xi'an 710049, China}
\email{fhyuan1@gmail.com}
\begin{document}
\maketitle
\vspace{-1em}
\begin{abstract}
A balanced $k$-partite $k$-graph is  a $k$-uniform
hypergraph such that every edge intersects each partition
class in exactly one vertex, where each partition class has size $n$.
Lo and Markstr\"om (2014)
determined the minimum vertex-degree threshold for perfect matchings in
balanced \(3\)-partite \(3\)-graphs.  In this paper, we determine the
minimum vertex-degree threshold for balanced \(4\)-partite \(4\)-graphs. The proof relies on a reduction framework for \(k\)-partite \(k\)-graphs, through which the existence of a perfect fractional matching is converted into a finite-dimensional optimization problem. 
\end{abstract}

\section{Introduction}

A matching in a hypergraph is a collection of pairwise disjoint edges.  One of
the central problems in extremal hypergraph theory is to determine which
degree conditions force the existence of a large matching or a perfect matching.
For graphs, a celebrated result of Dirac~\cite{dirac1952some} states that a graph on $n$ vertices has a perfect matching if every vertex has degree at least $n/2$.  For uniform hypergraphs, however, the problem becomes significantly
more delicate, since one can consider different types of degree conditions and new extremal constructions may appear.  

A $k$-uniform hypergraph (or $k$-graph for short) is a pair $H=(V,E)$ where $V$ is a finite set of vertices and $E\subseteq \binom{V}{k}$ is a family of $k$-element subsets of $V$, called edges. A matching in $H$ is a set of vertex-disjoint edges, and it is perfect if it covers all vertices of $H$.

For a set $S$ of vertices, the degree of $S$ in $H$, denoted by $d_H(S)$, is the number of edges containing $S$.  For an integer $\ell$ with $0\le \ell\le k-1$, the minimum $\ell$-degree $\delta_\ell(H)$ of $H$ is the minimum of $d_H(S)$ over all $\ell$-sets $S$ of vertices.  In particular, $\delta_0(H)=e(H)$ is the number of edges in $H$, and $\delta_1(H)=\delta(H)$ is the minimum vertex degree.
We define $m_l(k,n)$ to be the smallest integer $m$ such that every $k$-graph on $n$ vertices with $\delta_l(H)\ge m$ contains a perfect matching. 

For $\ell=0$, the problem is closely related to the classical Erd{\H{o}}s
Matching Conjecture~\cite{erdos_1965}.  For integers $k\ge 2$ and $s\ge 1$, let
$\operatorname{ex}_k(n,s)$ denote the maximum number of edges in a $k$-graph
on $n$ vertices with matching number less than $s$.  Erd{\H{o}}s conjectured
that
\[
        \operatorname{ex}_k(n,s)
        =
        \max\left\{
        \binom{ks-1}{k},
        \binom{n}{k}-\binom{n-s+1}{k}
        \right\}.
\]
The conjecture has been studied extensively.  The case $k=2$ follows from a classical
theorem of Erd{\H{o}}s and Gallai~\cite{erdos_gallai_1959}.
For the general case, this conjecture is true when $n$ is sufficiently large compared to $s$ and $k$; see
\cite{bollobas_daykin_erdos_1976,frankl_2017_maximum,MR2912790}.  More recently, Frankl and Kupavskii~\cite{fk_concentration_2022} proved the
conjecture in a wider range, in particular when $ n\ge \frac{5}{3}sk-\frac{2}{3}s$
and $s$ is sufficiently large.  Notable progress has also been achieved in several important cases.
Frankl~\cite{frankl_2017_maximum} settled the case $k=3$, while Frankl, Lu, Ma and Wu~\cite{frankl_lu_ma_wu_2026} established the corresponding $4$-uniform result in the range $n\ge 5s$ for all sufficiently large $n$.
The almost perfect matching case was resolved by Kolupaev and Kupavskii~\cite{kolupaev_kupavskii_2023}.

For other degree conditions, that is, for $1\le \ell\le k-1$, the threshold
is governed by two standard obstructions: the space barrier, consisting of all
edges intersecting a fixed set of size slightly smaller than $n/k$, and the
divisibility or parity barrier, which contains half of all the possible edges.  These lead to the following conjecture (see \cite{han_person_schacht_2009})
\[
        m_\ell(k,n)
        \sim
        \max\left\{
        \frac12,\,
        1-\left(1-\frac1k\right)^{k-\ell}
        \right\}
        \binom{n-\ell}{k-\ell}.
\]
The codegree case $\ell=k-1$ was determined exactly, for every fixed
$k\ge 3$ and sufficiently large $n$, by R{\"o}dl, Ruci{\'n}ski and
Szemer{\'e}di~\cite{rodl_rucinski_szemeredi_2009}.  For
$k/2\le \ell<k$, Pikhurko~\cite{pikhurko_2008} proved the asymptotically sharp threshold $m_\ell(k,n)\sim \frac12\binom{n-\ell}{k-\ell}$,
which was later sharpened to exact thresholds by Treglown and
Zhao~\cite{treglown_zhao_2012,treglown_zhao_2013}.  The range
$\ell<k/2$ is more delicate.  For vertex degree, the asymptotically sharp
threshold for $3$-graphs was obtained by H{\`a}n, Person and
Schacht~\cite{han_person_schacht_2009}, and the exact threshold was later
determined independently by K{\"u}hn, Osthus and
Treglown~\cite{kuhn_osthus_treglown_2013} and by Khan~\cite{khan_2013}; Khan
also determined the exact vertex-degree threshold for
$4$-graphs~\cite{khan_2016}.  General connections with the Erd{\H{o}}s
matching problem and fractional matching thresholds were developed in
\cite{alon_frankl_huang_rodl_rucinski_sudakov_2012,kuhn_osthus_townsend_2014,treglown_zhao_2016}.

\subsection{Matchings in $k$-partite $k$-graphs}
A $k$-partite $k$-graph is a $k$-uniform hypergraph whose vertex set is partitioned
into classes $V_1,\ldots,V_k$ and every edge contains exactly one vertex from each
class.  We say that it is balanced if
$
        |V_1|=\cdots=|V_k|=n .
$
For $0\le \ell\le k-1$, let $\delta_\ell(H)$ denote the minimum $\ell$-degree of
$H$.  More precisely, $\delta_\ell(H)$ is the minimum of $d_H(S)$ over all
$\ell$-sets $S$ which intersect every partition class in at most one vertex.  Define
$m'_\ell(k,n)$ to be the smallest integer $m$ such that every balanced
$k$-partite $k$-graph $H$ with
$\delta_\ell(H)\ge m$
contains a perfect matching.

Kühn and Osthus~\cite{kuhn_osthus_2006} proved a Dirac-type theorem for
perfect matchings in dense hypergraphs (including the balanced partite setting), 
showing that a codegree condition of order $n/2+\sqrt{2n\log n}$ forces a perfect matching.  Aharoni, Georgakopoulos and
Sprüssel~\cite{aharoni_georgakopoulos_sprussel_2009} later proved that $m_{k-1}'(k,n)\le n/2+1$. Given a set $L\in\binom{[k]}{\ell}$, an $\ell$-set $T\in\binom{V}{\ell}$ is an $L$-tuple if $|T\cap V_i|=1$ for all $i\in L$, and let $\delta_L(H)=\min d_H(T)$ over all $L$-tuples $T$.  Pikhurko~\cite{pikhurko_2008} proved that if $L\subseteq [k]$ with $|L|=\ell$, and
$
        \frac{\delta_L(H)}{n^{k-\ell}}
        +
        \frac{\delta_{[k]\setminus L}(H)}{n^\ell}
        \ge 1+o(1),
$
then the balanced $k$-partite $k$-graph $H$ contains a perfect matching.
Consequently, it implies that for $k/2\le \ell<k$, $m'_\ell(k,n)\sim \frac12 n^{k-\ell}$.
For the vertex-degree problem, Lo and Markström~\cite{lo_markstrom_2014} determined $m'_1(3,n)$.
For codegree thresholds, Lu, Wang and Yu~\cite{lu_wang_yu_2019} characterized the extremal balanced
$k$-partite $k$-graphs with $\delta_{k-1}(H)\ge n/2$ and no perfect matching,
thereby giving a sharp answer to a question of Rödl and Ruciński. 
For almost perfect
matchings, Han, Zang and Zhao\cite{han_zang_zhao_2019}, and independently Lu, Wang and Yu~\cite{lu_wang_yu_2018}, determined
the minimum codegree threshold in balanced $k$-partite $k$-graphs.

\subsection{Extremal constructions and main result}  We follow the notation of Lo and Markstr{\"o}m~\cite{lo_markstrom_2014}.
Let $d_1,d_2,\dots,d_k,n$ be integers such that $0\le d_i\le n$. let $V_1,\ldots,V_k$ be pairwise disjoint vertex sets with $|V_i|=n$ for every
$i\in[k]$. For all $i\in[k]$, let $W_i$ be a subset of $V_i$ with $|W_i|=d_i$. We denote by $ H_k(n;d_1,\ldots,d_k)$
the balanced $k$-partite $k$-graph on $V_1\cup\cdots\cup V_k$ whose edge set
consists of all legal $k$-sets intersecting $W_1\cup\cdots\cup W_k$. Define $H_k(n;m)$ (where $m=\sum_{i=1}^k d_i$) to be $H_k(n;d_1,d_2,\dots,d_k)$ with $d_i=\lfloor(m+i-1)/k\rfloor$ for all $i\in[k]$. We also define $H_k^0(n;m)$ to be the $k$-partite $k$-graph obtained from  $H_k(n;m)$ by deleting all edges contained in $W_1\cup W_2\cup \cdots \cup W_k$.
Let
\[
\mathscr{H}_k\!\left(n;m;\left\lceil\frac{m}{k}\right\rceil\right)
=
\left\{
H_k(n;d_1,\ldots,d_k):
\sum_{i=1}^k d_i=m
\text{ and }
d_i\le \left\lceil\frac{m}{k}\right\rceil
\text{ for all } i\in[k]
\right\}.
\]

Write $ m=rk+s$, where $r\ge 0$ and $s\in\{1,\dots,k\}$ are integers,
and put $ t=\left\lfloor m/(r+1)\right\rfloor$.
Define integers $a_1,\ldots,a_k$ by
\[
        a_i=
        \begin{cases}
        r+1, & 1\le i\le t,\\
        m-(r+1)t, & i=t+1,\\
        0, & t+2\le i\le k,
        \end{cases}
\]
where the middle line is omitted if $t=k$.  Let
$
        H'_k(n;m)=H_k(n;a_1,\ldots,a_k),
$
and set
$
        d_k(n,m)=\delta_1\bigl(H'_k(n;m)\bigr).
$
A direct calculation gives
\begin{equation}\label{eq:dk}
    d_k(n,m)=n^{k-1}-n^{\max\{0,k-t-1\}}(n-r-1)^{t-1}(n-m+(r+1)t)^{\min\{1,k-t\}}.
\end{equation}
When $k=4$ and $m=n-1$,
\eqref{eq:dk} yields, for $n\ne2,3,6$,
\begin{equation*}
    d_4(n,n-1)=
    \begin{cases}
    n^3-(\frac{3n}{4})^2(\frac{3n}{4}+1),
        & \text{if } n\equiv0\pmod{4},\\[2mm]
    n^3-(\frac{3n}{4}+\frac{1}{4})^3,
        & \text{if } n\equiv1\pmod{4},\\[2mm]
    n^3-(\frac{3n}{4}-\frac{1}{2})^2(\frac{3n}{4}+\frac{5}{2}),
        & \text{if } n\equiv2\pmod{4},\\[2mm]
    n^3-(\frac{3n}{4}-\frac{1}{4})^2(\frac{3n}{4}+\frac{7}{4}),
        & \text{if } n\equiv3\pmod{4}.
    \end{cases}
\end{equation*}

In this paper, we determine $m'_1(4,n)$ for all sufficiently large $n$.
\begin{theorem}\label{main_theorem}
   There exists an integer $n_0$ such that for every integer $n\ge n_0$, we have $m'_1(4,n)=d_4(n,n-1)+1$.
\end{theorem}

The lower bound is given by the extremal construction above.  Indeed,
$H'_4(n;n-1)$ has minimum vertex degree $d_4(n,n-1)$ and contains no perfect
matching. The proof of Theorem~\ref{main_theorem} is divided into two parts, according to whether the given hypergraph
is close to $H_4^0(n;n)$. For $\alpha>0$, we say that a balanced $4$-partite $4$-graph $H$ is
$\alpha$-close to $H_4^0(n;n)$ if
$
        |E(H_4^0(n;n))\setminus E(H)|\le \alpha n^4.
$ 
In this paper, we write \(a=b\pm c\) to mean
$b-c\le a\le b+c$.

\subsection*{Organization}
In Section~2, we develop the main reduction method of the paper. More precisely, we reduce the non-extremal fractional matching problem to a weighted stability problem for 3-partite 3-graphs, and then further reduce it, via the half-duplication procedure, to a finite-dimensional optimization problem. In Section~3, we prove the non-close case. In Section~4, we prove the close case using a stability result for small matchings developed by Lo and Markström~\cite{lo_markstrom_2014}. Finally, in Section~5, we combine the close and non-close cases to prove Theorem~\ref{main_theorem}.

\section{Perfect fractional matchings}
In this section, we prove the non-extremal fractional matching result (Lemma~\ref{thm:frac-mat}).  The
main ingredient is a stability lemma (Lemma~\ref{lem:frac-3-mat0}) for $3$-partite $3$-graphs.

Let $H$ be a $k$-graph.  A\textit{ fractional matching} of $H$ is a function
$f:E(H)\to[0,1]$ such that $\sum_{e\ni v}f(e)\le 1$ for every $v\in V(H)$,
and its size is $f(E(H))=\sum_{e\in E(H)}f(e)$. A fractional matching of maximum possible size is called a maximum fractional matching of $H$.  It is perfect if
$f(E(H))=|V(H)|/k$.  A \textit{fractional vertex cover} of $H$ is a function
$g:V(H)\to[0,1]$ such that $\sum_{v\in e}g(v)\ge 1$ for every $e\in E(H)$,
and its size is $g(V(H))=\sum_{v\in V(H)}g(v)$. For a subset $S\subseteq V(H)$, we write $g(S)=\sum_{v\in S}g(v)$. A fractional vertex cover of minimum possible size is called a minimum fractional vertex cover of $H$.

Let $H$ be a $k$-partite $k$-graph with vertex classes $V_1,\ldots,V_k$, and let $g:V(H)\to[0,1]$ be a fractional vertex cover of $H$. We say that $H$ is \emph{stable} with respect to $g$ if \[ E(H)= \left\{ e\subseteq V(H): |e\cap V_i|=1 \text{ for every } i\in[k], \text{ and } g(e)\ge 1 \right\}. \]
\begin{theorem}\label{thm:frac-mat}
Let $0<\gamma\ll\varepsilon\ll1$, and let $n$ be sufficiently large.
Let $H$ be a $4$-partite $4$-graph with vertex classes
$V_1,V_2,V_3,V_4$, each of size $n$. Suppose that
\[
        \delta_1(H)\ge \left(\frac{37}{64}-\gamma\right)n^3.
\]
If $H$ is not $\varepsilon$-close to $H_4^0(n;n)$, then $H$ contains a
perfect fractional matching.
\end{theorem}

\begin{lemma}\label{lem:frac-3-mat0}
Let $0<\gamma\ll \xi\ll1$, and let $n$ be sufficiently large.
Let $H$ be a $3$-partite $3$-graph with partition classes
$V_1,V_2,V_3$, each of size $n$, and let
$g:V(H)\to[0,1]$ be a fractional vertex cover of $H$. Suppose that
\begin{enumerate}[label=(\roman*)]
    \item $g(V(H))+\max \{g(V_i):i\in[3]\}\le n$; and
    %$2g(V_1)+g(V_2)+g(V_3)\le n$, $g(V_1)\ge \max\{g(V_2),g(V_3)\}$, and
    \item  $e(H)\ge \left(\frac{37}{64}-\gamma\right)n^3 .$
\end{enumerate}
Then there exist disjoint subsets $W_i,X_i\subseteq V_i$ for $i\in[3]$
with 
$|W_i|=(1/4\pm \xi)n$ and $|X_i|=(3/4\pm \xi)n$ such that $g(w)\ge 1-\xi$ for all $w\in W_i$ and $g(x)\le \xi$ for all $x\in X_i$.
\end{lemma}

We first prove Theorem~\ref{thm:frac-mat} using Lemma~\ref{lem:frac-3-mat0}, whose proof is postponed to the end of the section.
\subsection{Proof of Theorem~\ref{thm:frac-mat}}
Choose $\xi$ such that $\gamma\ll \xi\ll \varepsilon$.
Suppose, for a contradiction, that $H$ has no perfect fractional matching. Let $f$ be a maximum fractional matching of $H$, and let $g$ be a minimum
fractional vertex cover. By linear programming duality, we have $f(E(H))=g(V(H))<n$. Relabel the vertex classes so that $g(V_1)=\max\{g(V_i):i\in[4]\}$.
%Put $A_i=g(V_i)$ for $i\in[4]$ and relabel the vertex classes so that
%$A_1\ge A_2\ge A_3\ge A_4$.

We first find a vertex of weight zero in $V_1$. If $g(v)>0$ for every
$v\in V_1$, then complementary slackness gives
$\sum_{e\ni v}f(e)=1$ for every $v\in V_1$.
Since every edge contains exactly one vertex from $V_1$, this implies $f(E(H))=\sum_{v\in V_1}\sum_{e\ni v}f(e)=n$,
a contradiction. Hence there exists $v_0\in V_1$ with $g(v_0)=0$.

Let $ L=N_H(v_0)$
be the link $3$-graph of $v_0$ on $V_2\cup V_3\cup V_4$; that is,
$$E(L)=\Bigl\{\{x_2,x_3,x_4\}\in V_2\times V_3\times V_4:
\{v_0,x_2,x_3,x_4\}\in E(H)\Bigr\}.$$
Then the restriction of $g$ to $V_2\cup V_3\cup V_4$ is a fractional vertex
cover of $L$. Moreover, we have
$g(V(L))+\max\{g(V_i):2\leq i\leq 4\}\le g(V(H))<n$.
Also, we have
$e(L)=d_H(v_0)\ge \delta_1(H)
        \ge \left(37/64-\gamma\right)n^3$.
Now we can apply Lemma~\ref{lem:frac-3-mat0} to $L$ with vertex classes
$V_2,V_3,V_4$, we conclude that for every $i=2,3,4$, there
exist disjoint sets $W_i,X_i\subseteq V_i$ with 
$|W_i|=(1/4\pm \xi)n$ and $|X_i|=(3/4\pm \xi)n$ such that $g(w)\ge 1-\xi$ for all $w\in W_i$ and $g(x)\le \xi$ for all $x\in X_i$.

We next recover the corresponding structure inside $V_1$. Since the vertices in
$W_2,W_3,W_4$ have large $g$-weight, we have
\[
        g(V_2)+g(V_3)+g(V_4)
        \ge
        (1-\xi)(|W_2|+|W_3|+|W_4|)
        \ge
        \left(\frac34-6\xi\right)n .
\]
As $g(V(H))<n$, it follows that
$g(V_1)\le (1/4+6\xi)n$.
Define
$ W_1:=\{u\in V_1:g(u)\ge 1-3\xi\}$. Then
$|W_1|\le (1/4+7\xi)n$.

We now prove the corresponding lower bound on $|W_1|$. Fix a vertex
$x_4\in X_4$. For every choice of $u\in V_1\setminus W_1,\ x_2\in X_2$ and $x_3\in X_3$,
we have
$g(u)+g(x_2)+g(x_3)+g(x_4)<1$.
Thus $\{u,x_2,x_3,x_4\}\notin E(H)$, as $g$ is a fractional vertex cover of
$H$.  Consequently, $d_H(x_4)
        \le
        n^3-(n-|W_1|)|X_2||X_3|$.
On the other hand,
$d_H(x_4)\ge
        \left(37/64-\gamma\right)n^3 $.
Recall that $|X_i|=(3/4\pm \xi)n$ for $i=2,3$. Then
\[
        \left(\frac{37}{64}-\gamma\right)n^3
        \le
        n^3-(n-|W_1|)
        \left(\frac34-\xi\right)^2n^2 .
\]
Since $\gamma\ll \xi$, this implies $|W_1|\ge \left(1/4-7\xi\right)n$.
Thus we have $|W_1|=\left(1/4\pm 7\xi\right)n$.

We will now show that $H$ is close to $H_4^0(n;n)$, contradicting the
assumption. For each $i\in[4]$, choose a set
$W_i^0\subseteq V_i$  such that $|W_i^0\triangle W_i|\le 8\xi n$ and $W_i^0=\lceil (n-i+1)/4\rceil$.
Let
$U_i:=V_i\setminus W_i^0$ for each $i\in[4]$.
Let $H^0$  be the copy of $H_4^0(n;n)$ determined by
$W_1^0,W_2^0,W_3^0,W_4^0$, i.e.
\[
        E(H^0)
        =
        \Bigl\{
        \{v_1,v_2,v_3,v_4\}\in V_1\times V_2\times V_3\times V_4:\
        1\le
        \left|\{i:v_i\in W_i^0\}\right|
        \le 3
        \Bigr\}.
\]
It is enough to show $|E(H^0)\setminus E(H)|\leq 300\xi n^4$.

For this purpose, let
$Z_1=V_1\setminus W_1$ and, for $i=2,3,4$, put $ Z_i=X_i$.
Then for each $i\in[4]$, we have $|U_i\triangle Z_i|\le 20\xi n$.
Moreover, $Z_1\cup Z_2\cup Z_3\cup Z_4$ is independent in  $H$.
Indeed, if $z_i\in Z_i$ for all $i\in[4]$, then $ g(z_1)+g(z_2)+g(z_3)+g(z_4)<(1-3\xi)+3\xi=1$,
so $\{z_1,z_2,z_3,z_4\}\notin E(H)$. Also recall that $\delta_1(H)\ge (37/64-\gamma)n^3$. Hence,
for every $i\in[4]$ and every $w\in U_i\cap Z_i$, we have
\begin{align*}
  |N_{H^0}(w)\setminus N_H(w)|&\le d_{H^0}(w)-(d_H (w)-|N_H(w)\cap U_j\times U_k\times U_l|)\\
  &\le (\frac{37}{64}+\gamma)n^3-\left(\frac{37}{64}-\gamma\right)n^3+60\xi n^3\\
  &\le 61\xi n^3,
\end{align*}
where $\{j,k,l\}=[4]\setminus\{i\}$.

Therefore,
\begin{align*}
  |E(H^0)\setminus E(H)| &\le\sum_{i=1}^4\sum_{u\in U_i}|N_{H^0}(u)\setminus N_H(u)|\\
        &=\sum_{i=1}^4\sum_{u\in U_i\cap Z_i}|N_{H^0}(u)\setminus N_H(u)|+\sum_{i=1}^4\sum_{u\in U_i\setminus Z_i}|N_{H^0}(u)\setminus N_H(u)|\\
        &\le \sum_{i=1}^4\sum_{u\in U_i\cap Z_i}61\xi n^3+\sum_{i=1}^4\sum_{u\in U_i\setminus Z_i}n^3\\
        &\le  300\xi n^4.
\end{align*}
Since $\xi\ll \varepsilon$, $H$ is $\varepsilon$-close to
$H_4^0(n;n)$, a contradiction. This completes the proof of Theorem~\ref{thm:frac-mat}.
\qed

\subsection{A half-duplication operation}
We introduce an auxiliary class of weighted $3$-partite $3$-graphs.  % The parameter $m$ measures the amount of slack allowed.
\begin{definition}
    Let  $m\ge 0$ be an integer, and let $V_1,V_2,V_3$ be vertex sets with $|V_1|=|V_2|=|V_3|=n$. A pair $(H,g)$ is said to have property $\mathcal{S}_{n,m}$ if $H$ is a 3-partite 3-graph with partition classes $V_1,V_2,V_3$, and $g:V(H)\to[0,1]$ is a fractional vertex cover of $H$ such that 
    \begin{equation*}
    g(V(H))+\max\{g(V_i):i\in[3]\}\le n+2m.
        %2g(V_1)+g(V_2)+g(V_3)\le n+2m\qquad\text{and}\qquad g(V_1)+m\ge \max\{g(V_2),g(V_3)\}.
\end{equation*}
Define $f(n,m):=\max\{e(H):(H,g)\text{ has property }\mathcal{S}_{n,m}\text{ for some }g\}$.
\end{definition}

 \begin{proposition}\label{m+1_m}
For all sufficiently large $n$ and every fixed integer $0\le m\le n/100$, we have
\[
f(n,m+1)\le f(n,m)+8n^2.
\]
\end{proposition}
\begin{proof}
Let $(H,g)$ have property $\mathcal S_{n,m+1}$ such that $e(H)=f(n,m+1)$.
We may assume that $g(V(H))+\max\{g(V_i)\ |\ i\in[3]\}\ge n$, otherwise $(H,g)$ already has property $\mathcal S_{n,m}$ and thus $e(H)\le f(n,m)$. It follows that \(g(V_j)\ge n/4\) for some \(j\in[3]\);  we assume that \(j=1\). Thus there exists a set $S\subseteq V_1$ such that $|S|=8$ and $g(S)\ge 2$.
Now let $H'$ be obtained from $H$ by deleting all edges incident to $S$, and
define
\[
        g'(v)=
        \begin{cases}
        0, & v\in S,\\
        g(v), & v\notin S.
        \end{cases}
\]
Then $g'$ is a fractional vertex cover of $H'$. 
By the choice of $S$, we have
\[
g'(V(H'))+\max_{i\in[3]}g'(V_i)\le g(V(H))+\max_{i\in[3]}g(V_i)-2\le n+2m,
\]
 which shows that \((H',g')\) satisfies property \(\mathcal S_{n,m}\).
Consequently, \(e(H')\le f(n,m)\).  
Observe that $\sum_{x\in S}d_{H}(x)\leq |S|n^2\leq 8n^2$.  
Therefore,
\[
e(H)\le e(H')+\sum_{x\in S}d_H(x)\le f(n,m)+8n^2.
\]
This yields the desired inequality.
\end{proof}

Let \((H,g)\) be a pair consisting of a 3-partite 3-graph \(H\) with partition classes \((V_1,V_2,V_3)\) and a fractional vertex cover \(g\) of \(H\). Suppose that $n$ is even.
For each $i\in[3]$, we define a \textit{half-duplication} operation on \(V_i\) as follows.

Order $V_i=\{u_1,\dots,u_n\}$ so that $g(u_1)\ge g(u_2)\ge\cdots\ge g(u_n)$. 
For each $t\in\{1,\dots,n/2+1\}$, define $S_t:=\{u_t,u_{t+1},\dots,u_{t+ n/2-1}\}$.
Note that $g(S_{t+1})\le g(S_t)$ and $|g(S_{t+1})-g(S_{t})|\le 1$. Moreover,
$g(S_{1})\ge g(V_i)/2\ge g(S_{n/2+1})$. Thus there exists an index $t$ such that $|2g(S_{t})-g(V_i)|\le 1$.
Fix the smallest such index $t$, and let $V_i^1=S_t$ and $V_i^2=V_i\setminus S_t$. Let 
$\varphi:V_i^1\to V_i^2$ be  a bijection and  let $\psi:=\varphi^{-1}$.

We define two new $3$-partite $3$-graphs $H_1,H_2$ both on the same partition $(V_1,V_2,V_3)$, via neighborhood duplication as follows:
\[
N_{H_1}(x)=
\begin{cases}
N_H(x) & x\in V_i^1,\\
N_H(\psi(x)) & x\in V_i^2,
\end{cases}
\qquad
N_{H_2}(x)=
\begin{cases}
N_H(x) & x\in V_i^2,\\
N_H(\varphi(x)) & x\in V_i^1.
\end{cases}
\]
Correspondingly, we define fractional vertex covers $g_1$
 and $g_2$ by mirroring the weights on $V_i$ in the same way, while keeping all other weights unchanged:
\[
g_1(x)=
\begin{cases}
g(x) & x\in V(H)\setminus V_i^2,\\
g(\psi(x)) & x\in V_i^2,
\end{cases}
\qquad
g_2(x)=
\begin{cases}
g(x) & x\in V(H)\setminus V_i^1,\\
g(\varphi(x)) & x\in V_i^1.
\end{cases}
\]
Thus, we obtain two pairs $(H_1,g_1)$ and $(H_2,g_2)$. Let $\Gamma_{i,1}(H,g)$ be the pair with more edges; if $e(H_1)=e(H_2)$, set $\Gamma_{i,1}(H,g)=(H_1,g_1)$. Let $\Gamma_{i,2}(H,g)$ be the other pair. We write $e(\Gamma_{i,j}(H,g))$ for the number of edges in the
hypergraph of the pair $\Gamma_{i,j}(H,g)$.

\begin{proposition}\label{oper_prop}
    Let $m\ge 0$, and let $(H,g)$ have property $\mathcal{S}_{n,m}$. Then for every $i\in[3]$ and $j\in[2]$, $\Gamma_{i,j}(H,g)$ has property $\mathcal{S}_{n,m+1}$, and  $e(\Gamma_{i,1}(H,g))+e(\Gamma_{i,2}(H,g))=2e(H)$.
\end{proposition}
\begin{proof}
By construction, $g_1$ and $g_2$ are fractional
vertex covers of $H_1$ and $H_2$, respectively. Moreover, the half-duplication operation changes
the total weight on $V_i$ by at most $1$, and leaves the total weight of the other two
parts unchanged. Therefore both $\Gamma_{i,1}(H,g)$ and $\Gamma_{i,2}(H,g)$
have property $\mathcal S_{n,m+1}$.

Finally, if the operation is applied on $V_i$, then
$
        e(H_1)=2\sum_{x\in V_i^1}d_H(x),
$
and
$
        e(H_2)=2\sum_{x\in V_i^2}d_H(x).
$
Since every edge of $H$ contains exactly one vertex from $V_i$, we have
$
        e(H_1)+e(H_2)=2\sum_{x\in V_i}d_H(x)=2e(H).
$
This proves the proposition.
\end{proof}

For a fractional vertex cover \(h\) and \(i\in[3]\), write
$\mathrm{val}_i(h):=\bigl|\{h(x):x\in V_i\}\bigr|$,
namely, the number of distinct values of \(h\) on \(V_i\).
Now we describe an algorithm that outputs a vertex cover of only a few values. 

\begin{algorithm}[h]
\caption{half-duplication iterations}
\label{alg:Ai}
\SetKwInOut{Input}{Input}
\SetKwInOut{Output}{Output}
\Input{A pair \((H,g)\) and an index \(i\in[3]\).}
\Output{A pair \(\mathcal A_i(H,g)\).}

Let \(r:=\lceil \log_2 n\rceil\) and set \((J,h):=(H,g)\)\;

\For{\(s=1,\ldots,3r-1\)}
{
    \If{\(\mathrm{val}_i(h)\le 2\)}
    {
        \KwRet{\((J,h)\)}\;
    }

    Let $(J_1,h_1):=\Gamma_{i,1}(J,h)$ and $(J_2,h_2):=\Gamma_{i,2}(J,h)$.

    \If{\(\mathrm{val}_i(h_1)<\mathrm{val}_i(h)\)}
    {
        Replace \((J,h)\) by \((J_1,h_1)\)\;
    }
    \Else
    {
        Replace \((J,h)\) by \((J_2,h_2)\)\;
        \KwRet{\((J,h)\)}\;
    }
}

\KwRet{\((J,h)\)}\;
\end{algorithm}

\begin{proposition}\label{prop:two-values-one-side}
Fix \(i\in[3]\), let \(r:=\lceil\log_2 n\rceil\), and let
$(H',g'):=\mathcal A_i(H,g)$.
Then \(g'\) takes at most two distinct values on \(V_i\).
Moreover, \((H',g')\) is obtained from \((H,g)\) after at most \(3r-1\) half--duplication operations on \(V_i\).
\end{proposition}
\begin{proof}
We follow Algorithm~\ref{alg:Ai}.  Clearly the algorithm performs at most
\(3r-1\) half--duplication operations on \(V_i\).
If the
algorithm terminates with \(\mathrm{val}_i(h)\le 2\), then the conclusion is immediate.  Suppose
therefore that it terminates at a step where
\(\mathrm{val}_i(h_1)=\mathrm{val}_i(h)\), so that the algorithm outputs
\((J_2,h_2)\).  In the half--duplication step, one of the two sets
\(S_t\) and \(V_i\setminus S_t\) is copied to the other half.  If copying this
set does not decrease the number of values, then this set must meet every
current value-block, here a value-block means a maximal set of vertices of \(V_i\) on which \(h\) takes the same value.  Hence the other set meets at most two value-blocks.
Thus \(h_2\) takes at most two
values on \(V_i\).

 Thus we may assume that $(J_2,h_2)$ is never selected by the algorithm. Let \((H^{(s)},g^{(s)})\) be the pair obtained after \(s\) such operations.  After the first \(r\) operations, we claim that $\mathrm{val}_i(g^{(r)})\le 2r+1$.
Indeed, during each of these operations, the copied set is either an interval \(S_t\) or its complement in the ordering of \(V_i\) by the current cover values. Hence at most two current value-blocks are cut partially. Call a final value exceptional if it arises from a value-block that was cut partially in one of the first \(r\) operations. There are at most \(2r\) exceptional final values. Every non-exceptional final value is copied as a whole at each of the first \(r\) operations, and therefore has multiplicity at least \(2^r\). Since \(2^r\ge n=|V_i|\), there is at most one non-exceptional final value. This proves the claim.

After the first \(r\) operations, the algorithm perform at most
\(2r-1\) further operations.  Starting from at most
\(2r+1\) values, each
such operation strictly decreases \(\mathrm{val}_i\).
Therefore the final output \((H',g')=\mathcal A_i(H,g)\) satisfies $ \mathrm{val}_i(g')\le 2$.
This proves the proposition.
\end{proof}

\begin{proposition}\label{prop:edge-stability-one-side}
Let $n\in\mathbb{Z}$ be sufficiently large and $m\in\mathbb{Z}$ such that $m\le \sqrt{n}$. Fix \(i\in[3]\), let \(r=\lceil\log_2 n\rceil\), and let
$(H',g')=\mathcal A_i(H,g)$.
Assume that \((H,g)\) has property \(\mathcal{S}_{n,m}\) and
$e(H)\ge f(n,0)-\Delta$.
Then
$e(H')\ge f(n,0)-\Delta^\ast$, where
$\Delta^\ast:=2\Delta + 8(m+3r-1)n^2$.
\end{proposition}

\begin{proof}
Consider one iteration applied to a current pair $(J,h)$.  If the algorithm keeps
$\Gamma_{i,1}(J,h)$, then Proposition~\ref{oper_prop} gives
\[
        e(\Gamma_{i,1}(J,h))
        \ge
        \frac{
        e(\Gamma_{i,1}(J,h))+e(\Gamma_{i,2}(J,h))
        }2
        =
        e(J).
\]
Hence the number of edges does not decrease along every step in which the
procedure keeps the first branch.
If the algorithm never switches to \(\Gamma_{i,2}\), then
$e(H')\ge e(H)\ge f(n,0)-\Delta$,
and the conclusion follows.

Thus we may assume that the algorithm stops by switching to $\Gamma_{i,2}$ at
its final step.  Let $(J,h)$ be the pair immediately before this step, and
write
$(H',g')=\Gamma_{i,2}(J,h)$.
All previous steps kept the first branch, so
$e(J)\ge e(H)\ge f(n,0)-\Delta$.
Suppose that this final step is the $(s+1)$-st half-duplication operation in
the whole procedure.  Then \(s+1\le 3r-1\), and by Proposition~\ref{oper_prop}, both outputs of this final step have property \(\mathcal{S}_{n,m+s+1}\).
Therefore
$e\bigl(\Gamma_{i,1}(J,h)\bigr)\le f(n,m+s+1)$.
Using the edge identity from Proposition~\ref{oper_prop}, we obtain
\[
e(H')
=e\bigl(\Gamma_{i,2}(J,h)\bigr)
=2e(J)-e\bigl(\Gamma_{i,1}(J,h)\bigr)
\ge 2e(J)-f(n,m+s+1).
\]
Since $m+s+1\le m+3r-1\le n/100$, iterating Proposition~\ref{m+1_m} 
gives
\[
f(n,m+s+1)\le f(n,0) + 8(m+s+1)n^2 \le f(n,0) + 8(m+3r-1)n^2.
\]
Combining the last two inequalities with \(e(J)\ge f(n,0)-\Delta\), we get
\[
e(H')\ge 2(f(n,0)-\Delta)-\bigl(f(n,0)+8(m+3r-1)n^2\bigr)
= f(n,0)-2\Delta-8(m+3r-1)n^2.
\]
This is exactly the desired bound.
\end{proof}

\subsection{The reduced two-valued case}
We need the following lemma before proving Lemma~\ref{lem:frac-3-mat0}.
\begin{lemma}\label{lem:two-value}
There exist constants \(\varepsilon_0 > 0\) and \(C > 0\) such that for all \(0 < \varepsilon \le \varepsilon_0\) and all sufficiently large integers \(n\), the following holds.  
Let $(H,g)$ be a pair with property $\mathcal S_{n,\varepsilon n}$, and suppose
that $H$ is stable with respect to $g$ and $\mathrm{val}_i(g)\le 2$ for every $i\in[3]$. Then
\[
        e(H)\le \left(\frac{37}{64}+C\varepsilon\right)n^3.
\]
Moreover, if $e(H)\ge \left(37/64-\varepsilon\right)n^3$, then, for every $i\in[3]$, there are sets $V_i^+,V_i^-\subseteq V_i$ with $|V_i^+|=\left(1/4\pm C\varepsilon\right)n$, $|V_i^-|=\left(3/4\pm C\varepsilon\right)n$
such that $ g(v)\ge 1-C\varepsilon$ for $v\in V_i^+$ and $g(v)\le C\varepsilon$ for $v\in V_i^-$.
\end{lemma}
First, we prove an auxiliary lemma.
\begin{lemma}\label{appen1}
There exists a constant \(K > 0\) such that the following holds.  
Let \(t_1, t_2, t_3 \in [0, 1]\). Denote \(s = t_1 + t_2 + t_3\) and \(m = \max\{t_1, t_2, t_3\}\). If \(s + m \le 1 + \eta\), then  
\[
(1 - t_1)(1 - t_2)(1 - t_3) \ge \frac{27}{64} - K\eta.
\]  
Moreover, if in addition \((1 - t_1)(1 - t_2)(1 - t_3) \le \frac{27}{64} + \eta\), then for each \(i \in [3]\) we have \(t_i = \frac14 \pm K\eta\).
%There exists a constant $K>0$ such that the following holds. Let $t_1,t_2,t_3\in [0,1]$. Let $s=t_1+t_2+t_3$, and $m=\max\{t_1,t_2,t_3\}$. If $s+m\le 1+\eta$, then \[    (1-t_1)(1-t_2)(1-t_3)\ge \frac{27}{64}-K\eta. \] Moreover, if also $(1-t_1)(1-t_2)(1-t_3)\le 27/64+\eta$, then $t_i=1/4\pm K\eta$ or every $i\in[3]$.
\end{lemma}
\begin{proof}
Throughout the proof, \(K_1,\ldots,K_7\) denote positive absolute constants,
chosen sufficiently large so that the estimates below hold.
We first prove the lemma under the stronger assumption $s+m\le 1$. Then $0\le t_i\le m$ for every $i\in[3]$, and  $m\le 1/2$. Since $\log(1-x)/x$ is non-increasing on $(0,1)$, we have $\log(1-t_i)\ge t_i\log(1-m)/m$ for every $i\in[3]$. Hence $(1-t_1)(1-t_2)(1-t_3)\ge (1-m)^{s/m}$.

If $m\le 1/4$, note that $s/m\le 3$. Therefore $(1-t_1)(1-t_2)(1-t_3)\ge (1-m)^3\ge 27/64$. If $m\ge 1/4$, set $m'=(1-m)/m$. Since $s+m\le1$, we have $m\le 1/2$ and $s/m\le m'$. Therefore $1\le m'\le 3$. Thus $(1-t_1)(1-t_2)(1-t_3)\ge (m'/(m'+1))^{m'}\ge 27/64$. This proves the exact bound when $s+m\le 1$.

We next prove the stability statement in this case. Assume $s+m\le 1$ and $(1-t_1)(1-t_2)(1-t_3)\le 27/64+\delta$. First suppose $m\le 1/4$. Then
\[
    (1-m)^3\le (1-t_1)(1-t_2)(1-t_3)\le 27/64+\delta.
\]
It follows that $1/4-m\le K_1\delta$. Next suppose $m\ge 1/4$. With $m'=(1-m)/m$, we have
\[
    \left(\frac{m'}{m'+1}\right)^{m'}\le (1-t_1)(1-t_2)(1-t_3)\le \frac{27}{64}+\delta.
\]
The function $x\mapsto (x/(x+1))^x$ has derivative bounded away from $0$ on $[1,3]$. Hence $3-m'\le K_2\delta$, and consequently $m-1/4\le K_3\delta$. Thus, in all cases, $|m-1/4|\le K_4\delta$.

For each $i\in[3]$, write $r_i=m-t_i$. Then $r_i\ge0$ and $1-t_i=1-m+r_i$. We have
\[
    (1-t_1)(1-t_2)(1-t_3)\ge (1-m)^3+(1-m)^2(r_1+r_2+r_3).
\]
Using $(1-t_1)(1-t_2)(1-t_3)\le 27/64+\delta$ and $|m-1/4|\le K_4\delta$, we obtain $r_1+r_2+r_3\le K_5\delta$. Therefore $|t_i-1/4|\le K_6\delta$ for every $i\in[3]$.

Finally, assume only that $s+m\le 1+\eta$. If $s+m\le 1$, set $t_i'=t_i$. If $s+m>1$, set $t_i'=t_i/(s+m)$. Then $t_1'+t_2'+t_3'+\max_i t_i'\le 1$, and $|t_i-t_i'|\le \eta$ for every $i\in[3]$. Moreover,
\[
    \left|\prod_{i=1}^3(1-t_i)-\prod_{i=1}^3(1-t_i')\right|
    \le |t_1-t_1'|+|t_2-t_2'|+|t_3-t_3'|
    \le 3\eta.
\]
Applying the exact bound for case $s+m\le 1$ to $t_1',t_2',t_3'$ gives $\prod_{i=1}^3(1-t_i)\ge 27/64-3\eta$. If also $\prod_{i=1}^3(1-t_i)\le 27/64+\eta$, then $\prod_{i=1}^3(1-t_i')\le 27/64+4\eta$. The stability statement already proved above gives $t_i'=1/4\pm K_7\eta$ for every $i\in[3]$. Since $|t_i-t_i'|\le \eta$, the desired conclusion follows for $K=K_7+1$.
\end{proof}
Now, we are ready to prove Lemma~\ref{lem:two-value}.
\begin{proof}[Proof of Lemma~\ref{lem:two-value}]
     We may assume throughout that $\mathrm{val}_i(g)=2$ for every $i\in[3]$, and that $H$ is stable with respect to $g$.
Let $V_1,V_2,V_3$ be the partition classes of $H$.
Write $V_i=V_i^-\cup V_i^+$, where $g$ is constant on both $V_i^-$ and $V_i^+$, and the value of $g$ on $V_i^-$ is smaller than the value on $V_i^+$.
The edge-type set of the pair $(H,g)$ is
\[
\mathcal T(H,g)
=
\Bigl\{(\sigma_1,\sigma_2,\sigma_3)\in\{-,+\}^3:
V_1^{\sigma_1}\times V_2^{\sigma_2}\times V_3^{\sigma_3}
\subseteq E(H)\Bigr\}.
\]
Equivalently, $\mathcal T(H,g)$ records which of the eight possible value-types occur as edges of $H$.

If $|\mathcal T(H,g)|=8$, then $H$ is the complete 3-partite $3$-graph, which is impossible in the present setting. If $|\mathcal T(H,g)|\le 6$, then, for each possible type set, we solve the corresponding optimization problem using 
{\textbf{SCIP}}, and obtain $e(H)<(5/9+1/100)n^3$ (for $\varepsilon\le 0.0005$); see the detailed verification at
\url{https://github.com/feihong0810/scip-verification-finite-types}. Hence it remains to consider the case $|\mathcal T(H,g)|=7$.

 Write $|V_i^+|=a_i n$, $g|_{V_i^-}=x_i$, $g|_{V_i^+}=x_i+d_i$, and $A_i=g(V_i)/n=x_i+a_i d_i$, where $d_i>0$ for every $i\in[3]$. Since $(H,g)$ has property $\mathcal S_{n,\varepsilon n}$, we have 
 \begin{align}\label{SCIP-eq1}
A_1+A_2+A_3+\max\{A_i:i\in[3]\} \le 1+2\varepsilon.
 \end{align}

Let $q=1-x_1-x_2-x_3$. Since $A_i\ge x_i$, we have
\[
    1+2\varepsilon
    \ge A_1+A_2+A_3+\max\{A_i:i\in[3]\}
    \ge \frac43(x_1+x_2+x_3).
\]
Thus $q\ge 1/4-2\varepsilon$. By taking $\varepsilon_0$ sufficiently small, we may assume $q\ge 1/5$.

Since $|\mathcal T(H,g)|=7$, we have $x_1+x_2+x_3+d_i=1-q+d_i\ge 1$ for every $i\in[3]$, and so $d_i\ge q$ for every $i\in[3]$.
For every $i\in[3]$, let $\mu_i=a_i d_i/q$. By (\ref{SCIP-eq1}), we have
\[
    \sum_{i\in[3]}\frac{x_i}{q}+\sum_{i\in[3]}\mu_i
    +\max\{\left(\frac{x_i}{q}+\mu_i\right):i\in[3]\}
    \le \frac{1+2\varepsilon}{q}.
\]
Also, since $1/q=1+\sum_{i\in[3]} x_i/q$ and $q\ge 1/5$, we obtain 
\[
\mu_1+\mu_2+\mu_3+\max\{(\frac{x_i}{q}+\mu_i):i\in[3]\}\le 1+10\varepsilon. 
\]
 Since $d_i\ge q$, we have $a_i\le \mu_i$ for every $i\in[3]$. Therefore $a_1+a_2+a_3+\max\{a_i:i\in[3]\} \le 1+10\varepsilon$. Applying Lemma~\ref{appen1} to $a_1,a_2,a_3$, we get $(1-a_1)(1-a_2)(1-a_3)\ge 27/64-L_1\varepsilon$ for some absolute constant $L_1$.

Since  no edge of $H$ lies in $V_1^-\times V_2^-\times V_3^-$, we have
\[
    \frac{e(H)}{n^3}
    \le 1-(1-a_1)(1-a_2)(1-a_3)
    \le \frac{37}{64}+L_1\varepsilon.
\]
This proves the upper bound.

It remains to prove the stability statement. Assume that $e(H)\ge (37/64-\varepsilon)n^3$. Then the preceding estimate gives $(1-a_1)(1-a_2)(1-a_3)\le 27/64+\varepsilon$. By Lemma~\ref{appen1}, we have $a_i=1/4\pm L_2\varepsilon$ for every $i\in[3]$ and for some absolute constant $L_2$.

Recall that $\mu_1+\mu_2+\mu_3+\max\{(x_i/q+\mu_i):i\in[3]\}\le 1+10\varepsilon$, $a_i\le\mu_i$, and $q\ge1/5$. Since $a_i=1/4\pm L_2\varepsilon$, it follows that $\mu_i=1/4\pm L_3\varepsilon$, and $x_i\le L_4\varepsilon$ for every $i\in[3]$. In particular, $q\ge 1-3 L_4\varepsilon$. Moreover, from $\mu_i=1/4\pm L_3\varepsilon$ and $a_i=1/4\pm L_2\varepsilon$, we have  $x_i+d_i=x_i+\mu_iq/a_i\ge 1-L_5\varepsilon$ for every $i\in[3]$.

Taking $C=\max\{L_2,L_4,L_5\}$, we obtain $|V_i^+|=(1/4\pm C\varepsilon)n$, $|V_i^-|=(3/4\pm C\varepsilon)n$, $g(v)\ge 1-C\varepsilon$ for every $v\in V_i^+$, and $g(v)\le C\varepsilon$ for every $v\in V_i^-$. This completes the proof.
\end{proof}
\subsection{Proof of Lemma~\ref{lem:frac-3-mat0}}
We first use Lemma~\ref{lem:two-value} to obtain the extremal value of
$f(n,0)$.
\begin{lemma}\label{extre_value}
    For every $\gamma>0$ and all sufficiently large $n$, we have $f(n,0)\le (37/64+\gamma)n^3$.
\end{lemma}
\begin{proof}
We first assume that $n$ is even.  The case where $n$ is odd follows by
deleting one vertex from each class, applying the even case to the remaining
balanced graph of class size $n-1$.

Choose $0<\varepsilon\ll\gamma$, and let $(H,g)$ be a pair with property
$\mathcal S_{n,0}$ such that $e(H)=f(n,0)$.  We may assume that $H$ is stable with respect to
$g$.
Let $r=\lceil\log_2 n\rceil$, and apply algorithms
$\mathcal A_1,\mathcal A_2,\mathcal A_3$ successively.  Let
\[
        (H',g'):=\mathcal A_3(\mathcal A_2(\mathcal A_1(H,g))).
\]
By Proposition~\ref{prop:two-values-one-side}, the function $g'$ takes at most
two values on each vertex class.  Moreover, $H'$ is stable with respect to
$g'$, and $(H',g')$ has property
$\mathcal S_{n,9r}$.

By Proposition~\ref{prop:edge-stability-one-side}, applied three times, we have
$
        e(H')\ge f(n,0)-C_1rn^2,
$ for some constant $C_1$.
Since $9r\le \varepsilon n$ for sufficiently large $n$, Lemma~\ref{lem:two-value}
gives
$
        e(H')\le \left(37/64+C_2\varepsilon\right)n^3,
$ for some constant $C_2$.
Therefore
$
        f(n,0)
        \le e(H')+C_1rn^2
        \le \left(37/64+C_2\varepsilon\right)n^3+C_1rn^2.
$
Since  $rn^2\ll n^3$ and $\varepsilon\ll\gamma$, the desired bound follows.
\end{proof}

\begin{proof}[Proof of Lemma~\ref{lem:frac-3-mat0}]
    We may assume that $n$ is even and $H$ is stable with respect to $g$.
Choose a constant $\eta$ such that
$\gamma \ll \eta \ll \xi$
and put $r=\lceil \log_2 n\rceil$.
By Lemma~\ref{extre_value}, applied with parameter $\eta$, we have $f(n,0)\le (37/64+\eta)n^3$.
Since $(H,g)$ has property $\mathcal{S}_{n,0}$ and $ e(H)\ge \left(37/64-\gamma\right)n^3$,
it follows that $e(H)\ge f(n,0)-2\eta n^3$.

Fix $i\in[3]$, and let $j,k$ be the other two indices. Let $(J,h):=\mathcal A_j(\mathcal A_k(H,g))$ and $(P,q):=\mathcal A_i(J,h)$.
 $J$ and $P$ are stable with respect to $h$ and $q$, respectively. By Proposition~\ref{prop:two-values-one-side},
$(P,q)$ and $(J,h)$ has property $\mathcal S_{n,9r}$ and $\mathcal S_{n,6r}$, respectively. By
Proposition~\ref{prop:edge-stability-one-side}, we have
$e(P)\ge \left(37/64-C\eta\right)n^3$ and $e(J)\ge \left(37/64-C'\eta\right)n^3$,
for some absolute constant $C$ and $C'$.

By Proposition~\ref{prop:two-values-one-side}, $q$ takes at most
two values on each vertex class. Since $9r\le \eta n$ for
sufficiently large $n$, we can apply Lemma~\ref{lem:two-value} to
$(P,q)$.  Therefore, for every $s\in[3]$, there is a decomposition $V_s=V_s^-\cup V_s^+$
with
$|V_s^+|=\left(1/4\pm C\eta\right)n$, and $ |V_s^-|=\left(3/4\pm C\eta\right)n$ such that $q(v)\ge 1-C\eta$ for all $v\in V_s^+$ and $q(v)\le C\eta$ for all $v\in V_s^-$.
In particular, for every $s\in[3]$,
\begin{equation*}
  \label{i,j,k}
  q(V_s)\ge \left(1-C\eta\right)\left(\frac14-C\eta\right)n\ge \left(\frac14-2C\eta\right)n.
\end{equation*}

We now convert this structure back to the original cover $g$ on the fixed
side $V_i$.
Since the algorithm $\mathcal A_i$ only modifies the side $V_i$, we have $h(v)=q(v)$ for every $v\in V_j\cup V_k$.
Moreover, since the procedures $\mathcal A_j$ and $\mathcal A_k$ do not modify
the side $V_i$, we have $h(v)=g(v)$ for every $v\in V_i$.
%\begin{equation}\label{i_equal}
%  h(v)=g(v)\qquad \text{for every }v\in V_i.
%\end{equation}
Since $(J,h)$ has property $S_{6r}$,
we have $h(V(H))+\max_{i\in[3]}h(V_i)\le n+12r$.
Hence we have  $h(V_i)\le (1/4+7C\eta)n$.
%\begin{equation}
%  \label{i_sum}
%  h(V_i)\le (\frac{1}{4}+7C\eta)n.
%\end{equation}

Define $ W_i:=\{u\in V_i:h(u)\ge 1-2C\eta\}$.
Since $h(V_i)\le (1/4+7C\eta)n$, we have $|W_i|\le \left(1/4+8C\eta\right)n$.
On the other hand, for every
$u\in V_i\setminus W_i,\ x_j\in V_j^-$ and $x_k\in V_k^-$, we have
\[
        h(u)+h(x_j)+h(x_k)
        <
        (1-2C\eta)+C\eta+C\eta
        <1.
\]
Hence  \(ux_jx_k\notin E(J)\), and
\begin{equation*}
  \label{J_edge_up}
  e(J)\le n^3-|V_j^-||V_k^-||V_i\setminus W_i|\le n^3-(n-|W_i|)(\frac{3}{4}-C\eta)^2n^2.
\end{equation*}
Recall that $e(J)\ge (37/64-C'\eta)n^3$. Hence $|W_i|\ge (1/4-4C\eta-4C'\eta)n$.
Then by $h(V_i)\le (1/4+7C\eta)n$, we have
\begin{equation*}
  \label{remain_sum}
h(V_i\setminus W_i)\le (1/4+7C\eta)n-(1-2C\eta)(\frac{1}{4}-4C\eta-4C'\eta)n\le C''\eta n,
\end{equation*}
for some absolute constant $C''$.
Thus for all but at most $\sqrt{C''\eta}n$ vertices $v\in V_i\setminus W_i$,  we have $h(v)\le \sqrt{C''\eta}$. Let $X_i=\{v\in V_i\setminus W_i:h(v)\le \sqrt{C''\eta}\}$. Note that 
\[|V_i\setminus W_i|\ge n-\left(\frac14+8C\eta\right)n\ge (\frac34-8C\eta)n.\]
Recall $|W_i|\ge (1/4-4C\eta-4C'\eta)n$. Thus we have
\begin{equation*}
  \label{X_i_size}
  (\frac34-8C\eta-\sqrt{C''\eta})n\le |X_i|\le (\frac34+4C\eta+4C'\eta)n.
\end{equation*}

Since \(\eta\ll\xi\), then we have $|W_i|=(1/4\pm\xi)n$ and $|X_i|=(3/4\pm\xi)n$. Moreover, for every $v\in W_i$, we have $g(v)=h(v)\ge 1-2C\eta\ge 1-\xi$, and for every $v\in X_i$, we have $g(v)=h(v)\le \sqrt{C''\eta}\le \xi$. Since \(i\in[3]\) was arbitrary, the same argument gives the required sets
\(W_i,X_i\) for all three vertex classes.
\end{proof}

\section{Hypergraphs not close to $H^0_4(n;n)$}
In this section, we prove the non-close case, namely, we show that every balanced \(4\)-partite \(4\)-graph with minimum vertex degree exceeding \(d_4(n,n-1)\) contains a perfect matching, provided it is not close to \(H^0_4(n;n)\). 
We shall also use the following rainbow fractional matching theorem of
Aharoni, Holzman, and Jiang~\cite{AHJ}.
\begin{theorem}[Aharoni, Holzman, and Jiang,~\cite{AHJ}]\label{rainfracmat}
    Let $r\ge 2$ be an integer, and let $n$ be a positive rational number. Let
    $H_1,\dots,H_{\lceil rn\rceil}$ be $r$-graphs such that $\nu^*(H_i)\ge n$ for
    $i=1,\dots,\lceil rn\rceil$. Then there exist $e_1\in H_1,\dots,e_{\lceil rn\rceil}\in
    H_{\lceil rn\rceil}$ such that $\left\{e_1,\dots,e_{\lceil rn\rceil}\right\}$ has a fractional
    matching of size $n$.
\end{theorem}
The next lemma provides the absorbing matching used in the argument.
\begin{lemma}[Lo and Markstr\"om,~\cite{lo_markstrom_2014}]\label{absorb}
Let  $1 \leq \ell<k$, $0<\gamma< 1 /\left(10 k^{3}\right)$ and $\gamma^{\prime}=\gamma^{2 k-1} / 20$.
Then there is an integer $n_{0}$ such that for all $n>n_{0}$ the following holds:
Suppose $H$ is a $k$-partite $k$-graph with $n$ vertices in each class and minimum $\ell$-degree
$\delta_{\ell}(H) \geq(1 / 2+\gamma) n^{k-\ell}$.
Then there exists a matching $M$ in $H$ of size $|M| \leq(k-1) \gamma^{k} n$ such that, for every balanced set $W$ of size $|W| \leq k \gamma^{\prime} n$, there exists a matching covering exactly the vertices of $V(M) \cup W$.
\end{lemma}
To obtain an almost perfect matching in a sparse quasi-regular hypergraph, we
use the following theorem of Frankl and R\"odl~\cite{FR85} whose proof is based on R\"odl nibble technique. 
\begin{theorem}[Frankl and R\"odl,~\cite{FR85}]\label{nibble}
    For every integer $k\ge 2$ and any real $\sigma >0$, there exist $\tau = \tau(k,\sigma)$ and
    $d_0 = d_0(k,\sigma)$ such that for every $n\ge D\ge d_0$ the following holds:
    Every $n$-vertex $k$-graph $H$ with $(1-\tau)D<d_H(v)< (1+\tau)D$ for any $v\in V(H)$ and
    $\Delta_2(H)<\tau D$ contains a matching covering all but at most $\sigma n$ vertices.
\end{theorem}
We are now ready to prove the non-close case of
Theorem~\ref{main_theorem}.
\begin{theorem}\label{not_close_main}
Let $0<\varepsilon\ll 1$, and let $n$ be sufficiently large. Let $H$ be a
$4$-partite $4$-graph with vertex classes $V_1,V_2,V_3,V_4$, each of size
$n$. Suppose that $H$ is not $\varepsilon$-close to $H^0_4(n;n)$ and
\[
    \delta_1(H)>d_4(n,n-1).
\]
Then $H$ contains a perfect matching.
\end{theorem}

\begin{proof}
Choose constants satisfying
$
    1/n\ll \alpha\ll \gamma\ll \varepsilon\ll 1.
$
By the definition of $d_4(n,n-1)$, for sufficiently large $n$ we have
\[
    d_4(n,n-1)\ge \left(\frac{37}{64}-\frac{\gamma}{10}\right)n^3>\left(\frac12+\alpha\right)n^3.
\]
Applying Lemma~\ref{absorb} with $k_{\ref{absorb}}=4$, $\ell_{\ref{absorb}}=1$, and $\gamma_{\ref{absorb}} = \alpha$, we obtain a matching $M_{\rm abs}$ in $H$ with
$|M_{\rm abs}|\le 3\alpha^4 n$ such that for every balanced set 
$W\subseteq V(H)\setminus V(M_{\rm abs})$ with $|W|\le \alpha^7 n/5$,
there is a matching in $H$ covering exactly
$V(M_{\rm abs})\cup W$.

Let
$
    H':=H-V(M)$ and 
    $n':=n-|M|.
$
Then $H'$ is a balanced $4$-partite $4$-graph with $n'$ vertices in each
class. Since $|M_{\rm abs}|\le 3\alpha^4 n$, for every $v\in V(H')$ we have
\[
    d_{H'}(v)
    \ge d_H(v)-3|M_{\rm abs}|n^2
    \ge
    \left(\frac{37}{64}-\frac{\gamma}{5}\right)n'^3.
\]
Moreover, $H'$ is not $(\varepsilon/2)$-close to $H^0_4(n';n')$. Indeed, if
$H'$ is $(\varepsilon/2)$-close to $H_4(n';n')$, then since
$|M|=O(\alpha^4 n)$ and $\alpha\ll\varepsilon$, 
the original graph $H$ would be $\varepsilon$-close to $H^0_4(n;n)$, a contradiction.

\medskip
\noindent
\textbf{Claim 1.}
Let $t=\left\lfloor n'/\log n'\right\rfloor$.
There exist perfect fractional matchings $f_1,\dots,f_t$ in $H'$ such that,
if
$
    S_i:=\operatorname{supp}(f_i)=\{e\in E(H'): f_i(e)>0\},
$
then all of the followings hold:
\begin{enumerate}[label=(\roman*)]
    \item $S_1,\dots,S_t$ are pairwise disjoint;
    \item $|S_i|\le 4n'$ for every $i\in[t]$;
    \item for every pair $D\in \binom{V(H')}{2}$,
    $
        \sum_{i=1}^t
        \sum_{\substack{e\in E(H'), D\subseteq e}}
        f_i(e)
        \le 4.
    $
\end{enumerate}

\medskip
\noindent
\emph{Proof of Claim 1.}
%We construct the fractional matchings greedily. 
Suppose to the contrary that
$f_1,\dots,f_s$ are perfect fractional matchings satisfying (i), (ii), (iii) such that $s$ is maximum. So $0\le s<t$.

Define
\[
U_s:=
\left\{
D\in \binom{V(H')}{2}:
\sum_{i=1}^s
\sum_{\substack{e\in E(H')\\ D\subseteq e}}
f_i(e)>3
\right\},
\]
and let
$
E_s:=
\{e\in E(H'):\text{ there exists }D\in U_s\text{ such that }D\subseteq e\}.
$
Also put $F_s:=E_s\cup \bigcup_{i=1}^s S_i$. We shall show that $G_s:=H'-F_s$ still satisfies the assumptions of Theorem~\ref{thm:frac-mat}, and hence contains a
perfect fractional matching.

First fix a vertex $v\in V(H')$. For each $i\in[s]$, since $f_i$ is a perfect fractional matching, $\sum_{e\ni v} f_i(e)=1$. Hence
\[
    \sum_{w\in V(H')\setminus\{v\}}
    \sum_{\substack{e\in E(H')\\ \{v,w\}\subseteq e}}
    f_i(e)
    =3.
\]
Summing over $i\in[s]$, we have
\[\sum_{w\in V(H')\setminus\{v\}}\sum_{i=1}^s\sum_{\substack{e\in E(H')\\ \{v,w\}\subseteq e}}
    f_i(e)=\sum_{i=1}^s\sum_{w\in V(H')\setminus\{v\}}
    \sum_{\substack{e\in E(H')\\ \{v,w\}\subseteq e}}
    f_i(e)
    =3s
\]
So there are at most $s$ vertices $w$ such that
the pair $\{v,w\}$ belongs to $U_s$. Thus the number of edges of $E_s$
containing $v$ and containing a pair from $U_s$ which includes $v$ is at most $sn'^2$.

On the other hand, since each $4$-edge contains six pairs and each $f_i$ has size $n'$, we have
\[
    \sum_{D\in\binom{V(H')}{2}}
    \sum_{i=1}^s
    \sum_{\substack{e\in E(H')\\D\subseteq e}}
    f_i(e)
    =
    6sn',
\]
Hence $|U_s|\le 2sn'$.
Thus the number of edges of $E_s$ containing $v$
and containing a pair from $U_s$ not involving $v$ is at most $2sn'^2$.
Consequently,
$
    d_{E_s}(v)\le 3sn'^2$.
Moreover, since $|S_i|\le 4n'$ for $i\le s$, we have $d_{\bigcup_{i=1}^s S_i}(v)\le 4sn'$.
Thus
$
    d_{F_s}(v)\le 4sn'^2$.
Since $s<t\le n'/\log n'$, it follows that
\[
    d_{G_s}(v)
    \ge d_{H'}(v)-4sn'^2
    \ge
    \left(\frac{37}{64}-\frac{\gamma}{3}\right)(n')^3
\]
for sufficiently large $n$.

Next, $G_s$ is not $(\varepsilon/4)$-close to $H^0_4(n';n')$.
Indeed,
\[
    |F_s|
    \le \frac14\sum_{v\in V(H')} d_{F_s}(v)
    \le 4sn'^3
    \le \frac{4(n')^4}{\log n'}
    =o(\varepsilon (n')^4).
\]
Thus, if $G_s$ is $(\varepsilon/4)$-close to $H^0_4(n';n')$, then $H'$ would be
$(\varepsilon/3)$-close to $H^0_4(n';n')$, and hence $H$ would be
$\varepsilon$-close to $H^0_4(n;n)$, a contradiction.

Therefore Theorem~\ref{thm:frac-mat}, applied with parameters
$\gamma/3$ and $\varepsilon/4$, implies that $G_s$ has a perfect fractional
matching. Applying Theorem~\ref{rainfracmat}, we obtain a perfect fractional
matching $f_{s+1}$ in $G_s$ whose support
$
   S_{s+1}:=\operatorname{supp}(f_{s+1})
$
has size at most $4n'$, satisfying (ii).

Since $S_{s+1}\subseteq E(G_s)$, it is disjoint from
$S_1,\ldots,S_s$ and from $E_s$, satifying (i).
If $D\notin U_s$, then $$\sum_{i=1}^s
    \sum_{\substack{e\in E(H')\\D\subseteq e}}
    f_i(e)
    \le 3.$$
Since
$f_{s+1}$ is a fractional matching, the contribution of $f_{s+1}$ to any fixed pair $D$ is at most $1$. Hence the total contribution to \(D\) is
at most \(4\).
If $D\in U_s$, then every edge containing $D$ lies in $E_s$, and hence no
edge of $S_{s+1}$ contains $D$, satisying (iii).
This proves the claim.
%Therefore the same bound holds in this case as well. This completes the greedy step and proves the claim.
$\hfill\blacksquare$
\medskip

Define
$
    h:E(H')\to[0,1]
$    and
$    h(e):=\sum_{i=1}^t f_i(e).
$
Since the supports $S_1,\dots,S_t$ are pairwise edge-disjoint, we have
$h(e)\le 1$ for every $e\in E(H')$.
Let $H''$ be a random spanning subgraph of $H'$ obtained by choosing each edge
$e\in E(H')$ independently with probability $h(e)$. For every vertex
$v\in V(H')$, we have
\[
    \mathbb E[d_{H''}(v)]
    =
    \sum_{e\ni v} h(e)
    =
    \sum_{i=1}^t \sum_{e\ni v}f_i(e)
    =
    t.
\]
For every pair $D\in \binom{V(H')}{2}$, we have
\[
    \mathbb E[d_{H''}(D)]
    =
    \sum_{e\supseteq D}h(e)
    =
    \sum_{i=1}^t
    \sum_{\substack{e\in E(H')\\D\subseteq e}}
    f_i(e)
    \le 4.
\]
By Chernoff's bound and union bound, with
probability $1-o(1)$, $H''$ satisfies
\begin{align*}
    \left(1-\frac1{\log n'}\right)t
    \le d_{H''}(v)
    \le
    \left(1+\frac1{\log n'}\right)t
    \qquad
    \text{for all }v\in V(H'), \label{eq:regular-Hpp}
\end{align*}
and $d_{H''}(D)\le \sqrt{n'}$ for all $D\in \binom{V(H')}{2}$.

Fix such a choice of $H''$.
Choose \(\sigma>0\)  with $ 4\sigma n'\le \alpha^7 n/5$.
We apply Theorem~\ref{nibble} to \(H''\) with \(D=t\). For sufficiently large \(n\), Theorem~\ref{nibble} gives a matching
\(M_{\rm near}\) in \(H''\) covering all but at most $\sigma |V(H'')|=4\sigma n'$ vertices.

Let
$
    W:=V(H')\setminus V(M_{\rm near}).
$
Since $H'$ is balanced $4$-partite and every edge of $M_{\rm near}$ contains one vertex
from each class, the set $W$ is balanced. Moreover,
$
    |W|\le 4\sigma n'\le 4\alpha'n.
$
By the absorbing property of $M_{\rm abs}$, there exists a matching $M_{\rm fill}$ in $H$
covering exactly the vertices of $V(M_{\rm abs})\cup W$.
Since $M_{\rm near}$ is disjoint from $V(M_{\rm abs})\cup W$, the union
$
    M_{\rm near}\cup M_{\rm fill}
$
is a perfect matching of $H$. This completes the proof.
\end{proof}

\section{Hypergraphs close to $H^0_k(n;n)$}

In this section, we deal with the close case. The following lemma follows from Theorem~1.6 in \cite{lo_markstrom_2014}; Moreover, the authors established the case when $k=3$.
\begin{lemma}[Lo and Markstr\"{o}m, \cite{lo_markstrom_2014}]\label{partial_matching}
	Let \(k,m,n\) be integers such that \(k\ge3\), \(n\ge k^7m\), and
	\(m\ge (k-1)^2\).  Let \(H\) be a \(k\)-partite \(k\)-graph with each
	vertex class of size \(n\).  If
	\[
	\delta_1(H)>d_k(n,m),
	\]
	then \(H\) contains a matching of size \(m+1\).
\end{lemma}

Recall that \(V_1,\ldots,V_k\) are the partition classes of
\(H_k^0(n;n)\) and   for each \(i\in[k]\),  \(W_i\in  {V_i\choose \lfloor (n-i+1)/k \rfloor}\). 
Define
$$
W=\bigcup_{i=1}^k W_i \quad \text{and} \quad X_i=V_i\setminus W_i \;\text{for all }i\in [k].
$$
Observe that $|W|=n$. Moreover, every edge of $H_k^0(n;n)$ intersects $W$ but is not fully contained in $W$.
Now let $H$ be a $k$-partite $k$-graph with vertex set $V\big(H_k^0(n;n)\big)$.
A vertex \(v\in V(H)\) is called \textit{\(\alpha\)-bad }if
$
|N_{H_k^0(n;n)}(v)\setminus N_H(v)|>\alpha n^{k-1};
$ otherwise \(v\) is called \textit{\(\alpha\)-good.}

\begin{lemma}\label{alpha_good}
    Let $k\ge 4$ be an integer, and let $0<\alpha \le 1/(k!5^k)$. Suppose $H$ is a
$k$-partite $k$-graph on $V(H^0_k(n;n))$
and every vertex of $H$ is $\alpha$-good. Then $H$ contains a perfect matching.
\end{lemma}
\begin{proof}
    Let $M$ be a largest matching in $H$ consisting of edges with exactly one vertex in $W$.
If $|M|= n$ then $H$ has a perfect matching; so we may assume that $|M|<n$. Since \(|W|=n\), there exists a vertex \(w\in W\setminus V(M)\).  Suppose that
\(w\in V_j\).

We first show that $|M|> n/2$. Assume, to the contrary, that $|M|\le n/2$. Then $|X_i\setminus V(M)|\ge n/2-n/k-1$ for each $i\in[k]$. By the maximality of \(M\), no edge of \(H\) contains \(w\) together with one
vertex from each set \(X_i\setminus V(M)\), \(i\in[k]\setminus\{j\}\).  On the
other hand, every such legal set is an edge of \(H_k^0(n;n)\).  Therefore
\[
|N_{H_k^0(n;n)}(w)\setminus N_H(w)|
\ge
\prod_{i\in[k]\setminus\{j\}} |X_i\setminus V(M)|
\ge
(n/2-n/k-1)^{k-1}
>
\alpha n^{k-1},
\]
where the last inequality holds since
\(\alpha\le 1/(k!5^k)\).  This contradicts the assumption that \(w\) is
\(\alpha\)-good.  Hence $|M|> n/2$.

Next we choose vertices $v_i\in X_i\setminus V(M)$ for every $i\in[k]\setminus j$. Indeed, let
$
U_i^W=|W_i\setminus V(M)|,
U_i^X=|X_i\setminus V(M)|.
$
Since each edge of \(M\) contains exactly one vertex of \(W\), we have
$\sum_{i=1}^k U_i^W=n-|M|$.
Moreover, for each $i\in[k]$, we have
$U_i^W+U_i^X=n-|M|$.
As \(w\in W_j\setminus V(M)\), we have \(U_j^W>0\).  If \(U_i^X=0\) for some
\(i\neq j\), then \(U_i^W=n-|M|\), forcing all uncovered vertices of \(W\) to
lie in \(W_i\), a contradiction.  Thus \(U_i^X>0\) for every \(i\neq j\), as
required.

 Let $S=\{v_i:i\in[k]\setminus j\}\cup\{w\}$.
 Let \(Q\subseteq M\) be any \((k-1)\)-set of matching edges.  If $H[S\cup \bigcup_{e\in Q} e]$ contains a perfect matching consisting of edges with exactly one vertex
 in \(W\), then replacing the \(k-1\) edges of \(Q\) by this perfect matching
 would give a larger matching of the same type, contradicting the maximality of
 \(M\).  Hence, for every such \(Q\), at least one legal crossing edge from
 \(H_k^0(n;n)\) is missing from \(H\).
 More precisely, for every \((k-1)\)-set \(Q\subseteq M\), there exists an edge
 $
 f_Q\in E(H_k^0(n;n))\setminus E(H)
 $
 such that \(f_Q\) contains exactly one vertex from \(S\), exactly one vertex
 from each edge of \(Q\), and exactly one vertex of \(W\). 
 
 Since $|M|> n/2$, the number of choices for \(Q\) is
 $
 \binom{|M|}{k-1}>\binom{n/2}{k-1}.
 $
 Therefore more than \(\binom{n/2}{k-1}\) missing edges of \(H_k^0(n;n)\) meet
 \(S\).  Since \(|S|=k\), there exists \(u\in S\) such that
 \[
 |N_{H_k^0(n;n)}(u)\setminus N_H(u)|
 >
 \frac1k\binom{n/2}{k-1}
 \ge
 \alpha n^{k-1},
 \]
 contradicting that \(u\) is \(\alpha\)-good.  This  proves
 the lemma.
\end{proof}

\begin{theorem}\label{close_main}
    Let $k\ge 4$ be an integer, and let $0<\varepsilon<\min\{1/k^{20},1/(5^{2k+1}k!^2)\}$. The following holds for all
    sufficiently large $n$.  Let \(H\) be a \(k\)-partite \(k\)-graph on
    \(V(H_k^0(n;n))\).  Suppose that \(H\) is \(\varepsilon\)-close to
    \(H_k^0(n;n)\) and
    \[
    \delta_1(H)>d_k(n,n-1).
    \]
    Then \(H\) contains a perfect matching.
\end{theorem}

\begin{proof}
Let $\alpha=\sqrt{\varepsilon}$. 
Let \(B\) be the set of all \(\alpha\)-bad vertices.
Since $H$ is $\varepsilon$-close to $H^0_k(n;n)$, we have
 $|B|\le k\sqrt{\varepsilon}\,n$.

Write $n-1=rk+s$ with $1\le s\le k$, and let
$c=\lceil k\sqrt{\varepsilon}\,n\rceil$.
Choose a subset $ B_0\subseteq W$ such that $B\cap W\subseteq B_0$ and $|W_i\setminus B_0|=r-c$ for every $i\in[k]$. Let $W'=W\setminus B_0$ and $H'=H-W'$. Note that $\frac{rk+s}{r+1}\ge k-1$ and $\frac{ck+s}{c+1}\ge k-1$. Then by \eqref{eq:dk}, we have
\begin{align*}
    \delta_1(H')&\ge \delta_1(H)-\bigl(n^{k-1}-(n-r+c)^{k-1}\bigr)\notag\\
    &>d_k(n,rk+s)-\bigl(n^{k-1}-(n-r+c)^{k-1}\bigr)\notag\\
    &=\Bigl[n^{k-1}-(n-r-1)^{k-2}(rk-r+k)\Bigr]
  -\Bigl[n^{k-1}-(n-r+c)^{k-1}\Bigr]\notag\\
&=(n-r+c)^{k-1}-(n-r-1)^{k-2}(rk-r+k)\notag\\
&=d_k(n-r+c,kc+s).
\end{align*}
Since \(\varepsilon\le k^{-20}\) and \(n\) is sufficiently large, applying Lemma~\ref{partial_matching}
to \(H'\), we obtain a matching \(M_1\) in \(H'\) of size $kc+s+1$.

Let $H_1=H-V(M_1)$. For $v\in B\setminus V(M_1)$, let $v\in B_1$ if and only if \(H_1\)
contains at least \(\varepsilon^{1/4}n^{k-1}\) edges containing \(v\) and
meeting \(W'\) in exactly one vertex.  Let $B_2=B\setminus (V(M_1)\cup B_1)$.

Since \(|B_1|\le |B|\le k\sqrt{\varepsilon}n\) and \(\varepsilon^{1/4}n^{k-1}\gg k|B_1|n^{k-2}\), a greedy construction yields a matching \(M_2\) in \(H_1\) that covers \(B_1\) and has exactly one vertex in \(W'\) per edge.
Let $H_2=H_1-V(M_2)$. Next we cover \(B_2\) by a matching whose edges are disjoint from \(W'\). By the definition of $B_2$, for each $v\in B_2$, the number of edges containing $v$ and disjoint from $W'$ is at least
\begin{align*}
    &\delta_1(H)-k|M_1\cup M_2|n^{k-2}-\varepsilon^{1/4}n^{k-1}-\Bigl(n^{k-1}-(n-r+c)^{k-1}-(k-1)(r-c)n^{k-2}\Bigr)\\
    >&\Bigl(n^{k-1}-(n-r-1)^{k-2}(rk-r+k)\Bigr)-\Bigl(n^{k-1}-(n-r+c)^{k-1}-(k-1)(r-c)n^{k-2}\Bigr)-2\varepsilon^{1/4}n^{k-1}\\
    =&(n-r+c)^{k-1}+(k-1)(r-c)n^{k-2}-(n-r-1)^{k-2}(rk-r+k)-2\varepsilon^{1/4}n^{k-1}\\
    >&(n-r)^{k-1}+(k-1)(r-c)n^{k-2}-(n-r)^{k-2}(n-r+k)-2\varepsilon^{1/4}n^{k-1}\\
    >&\varepsilon^{1/4}n^{k-1}.
\end{align*}
Thus another greedy argument gives a
matching \(M_3\) in \(H_2\) which covers \(B_2\) and whose edges are disjoint
from \(W'\).

Let $H_3=H_2-V(M_3)$ and $n_1=n-|M_1\cup M_2\cup M_3|$.
It remains to adjust the number of remaining vertices of \(W'\).  Since
\(M_1\) and \(M_3\) are disjoint from \(W'\), and every edge of \(M_2\) meets
\(W'\) in exactly one vertex, we have
$
|W'\cap V(H_3)|
=
k(r-c)-|M_2|.
$
Moreover, since \(|M_1|=kc+s+1\), this quantity is at least
\(n_1+|M_3|\).

First, choose a appropriate matching $M_4$ in $H_3$ of size at most $|M_1|+|M_3|$, such that $M_4$ covers $B_0\cap V(H_3)$ and every edge of $M_4$ contains exactly two vertices from $W\cap V(H_3)$, and, writing $H_4=H_3-V(M_4)$, we have $|W\cap V(H_4)|=n_1-|M_4|$.  
Next  we choose a matching \(M_5\) in \(H_4\) of 
size at most \(2k|M_1\cup M_2\cup M_3|\), such that every edge of \(M_4\) meets \(W'\) in
exactly one vertex, and in $H_5:=H_4-V(M_5)$, we have $|W\cap V_i\cap V(H_5)|=\left\lfloor\frac{n_2+i-1}{k}\right\rfloor$
for every $i\in[k]$, where $n_2=n_1-|M_4\cup M_5|$.

For every \(v\in V(H_5)\), every edge of \(H_k^0(n_2;n_2)\) missing
from \(H_5\) is also an edge of \(H_k^0(n;n)\) missing from \(H\).  Therefore
\[
|N_{H_k^0(n_2;n_2)}(v)\setminus N_{H_5}(v)|
\le
|N_{H_k^0(n;n)}(v)\setminus N_H(v)|
\le
\alpha n^{k-1}.
\]
Since
$n_2\ge n-4k^3\sqrt{\varepsilon}\,n$,
we have
$\alpha n^{k-1}\le 2\alpha n_2^{k-1}$
for \(n\) sufficiently large.  Thus every vertex of \(H_5\) is
\(2\alpha\)-good.  Since
$
2\alpha=2\sqrt{\varepsilon}\le 1/{k!5^k}$,
we apply Lemma~\ref{alpha_good} to \(H_5\) and obtain a perfect matching \(M_6\) in
\(H_5\).  Consequently,
$
M_1\cup M_2\cup M_3\cup M_4\cup M_5\cup M_6
$
is a perfect matching of \(H\), as required.
\end{proof}

\section{Proof of Theorem~\ref{main_theorem}}

\begin{proof}
	The lower bound follows from the extremal construction \(H'_4(n;n-1)\), which
	has minimum vertex degree \(d_4(n,n-1)\) and contains no perfect matching.
	Now we prove the upper bound.  Let \(H\) be a balanced \(4\)-partite \(4\)-graph
	with vertex classes \(V_1,V_2,V_3,V_4\), each of size \(n\), and suppose that
	$
	\delta_1(H)>d_4(n,n-1).
	$
	Choose \(\varepsilon>0\) sufficiently small so that  Theorems~\ref{not_close_main} and \ref{close_main} apply.  If \(H\) is \(\varepsilon\)-close to
	\(H_4^0(n;n)\), Theorem~\ref{close_main}, applied with \(k=4\), gives a
	perfect matching in \(H\).  Otherwise, \(H\) is not \(\varepsilon\)-close to
	\(H_4^0(n;n)\), and Theorem~\ref{not_close_main} gives a perfect matching in
	\(H\).
	
	Thus every balanced \(4\)-partite \(4\)-graph \(H\) with
	\(\delta_1(H)>d_4(n,n-1)\) contains a perfect matching.  Together with
	the lower bound, this gives
	$
	m'_1(4,n)=d_4(n,n-1)+1,
	$
	as claimed.
\end{proof}
\section*{Acknowledgements}
Hongliang Lu is supported by the National Key R\&D Program of China (No.~2023YFA1010203). Yan Wang is supported by the National Key R\&D Program of China (No.~2022YFA1006400) and the National Natural Science Foundation of China
(No.~12571376).
\bibliographystyle{plain}  
\bibliography{refs}
\end{document}